\documentclass[11pt]{article}
\usepackage{geometry}                
\usepackage{graphicx}
\usepackage{amssymb}
\usepackage{epstopdf}
\usepackage[T1]{fontenc}
\usepackage[ngerman, english]{babel}
\usepackage{mathtools}
\usepackage{bbm}
\usepackage{mathrsfs}
\usepackage{amsmath,amscd}
\usepackage{tikz}
\usepackage{tikz-cd}
\usepackage{babel}
\usetikzlibrary{arrows, matrix}
\usetikzlibrary{cd}
\usepackage[arrow, matrix, curve]{xy}

\newcommand{\qed}{\hfill \ensuremath{\Box}}
\newenvironment{proof}{\vspace{1ex}\noindent{\it Proof.}\hspace{0.5em}}
	{\hfill\qed\vspace{1ex}}

\newtheorem{theorem}{Theorem}[section]
\newtheorem{lemma}[theorem]{Lemma}
\newtheorem{proposition}[theorem]{Proposition}
\newtheorem{corollary}[theorem]{Corollary}
\newtheorem{definition}[theorem]{Definition}

\DeclareMathOperator{\Gal}{\operatorname{Gal}}
\DeclareMathOperator{\Q}{\mathbf{Q}}

\DeclareMathOperator{\Z}{\mathbf{Z}}

\DeclareMathOperator{\A}{\mathbf{A}}

\DeclareMathOperator{\N}{\mathbf{N}}

\DeclareMathOperator{\Spec}{\operatorname{Spec}}

\DeclareMathOperator{\Hom}{\operatorname{Hom}}

\DeclareMathOperator{\Pic}{\mathrm{Pic}}

\DeclareMathOperator{\Gr}{\mathrm{Gr}}
\DeclareMathOperator{\et}{\acute{\mathrm{e}}{\mathrm{t}}}
\DeclareMathOperator{\Gm}{\mathbf{G}_m}

\DeclareMathOperator{\Res}{\mathrm{Res}}

\title{Finite descent obstruction and non-abelian reciprocity}
\author{Otto Overkamp \footnote {Imperial College London, South Kensington Campus, London SW7 2AZ; \newline otto.overkamp13@imperial.ac.uk}}
\date{}
\begin{document}
\maketitle
{\abstract{For a nice algebraic variety $X$ over a number field $F$, one of the central problems of Diophantine Geometry is to locate precisely the set $X(F)$ inside $X(\A_F)$, where $\A_F$ denotes the ring of adèles of $F$. One approach to this problem is provided by the finite descent obstruction, which is defined to be the set of adelic points which can be lifted to twists of torsors for finite étale group schemes over $F$ on $X$. More recently, Kim proposed an iterative construction of another subset of $X(\A_F)$ which contains the set of rational points. In this paper, we compare the two constructions. Our main result shows that the two approaches are equivalent.}}\\
\\
$\mathbf{Keywords}$: Diophantine Geometry, Finite descent obstruction, Adelic points, Rational points\\
$\mathbf{Mathematics \; Subject \; Classification}$: 11G35\\
\tableofcontents
\section*{Introduction}
Let $F$ be a finite extension of the field $\Q$ of rational numbers. Let $X$ be a geometrically integral smooth quasi-projective algebraic variety over $F$ with a rational point $b\in X(F)$, which we will keep fixed from now on. Let $\overline{F}$ denote an algebraic closure of $F$ and put $\overline{X}:=X\times_F\Spec\overline{F}$. Let us first describe the main ideas of this paper informally. Denote by $\A_F$ the ring of ad$\grave{\mathrm{e}}$les of $F$. One central question of Diophantine geometry is to locate precisely the set $X(F)$ inside $X(\A_F).$ There are several \it local-to-global principles \rm that attempt to rectify the situation, for example the (finite) descent obstruction, introduced by Harari and Skorobogatov \cite{HSk}. This is a set $X(\A_F)^{\text{f-cov}}$
such that 
$$X(F)\subseteq X(\A_F)^{\text{f-cov}}\subseteq X(\A_F).$$
It is defined using finite $\et$ale covers of $X$; a precise definition will be recalled later. 
On the other hand, we can construct similar sets lying between $X(F)$ and $X(\A_F)$ using ideas from homotopy theory and the geometric $\et$ale fundamental group $\pi_1(\overline{X}, b)$ of $X$. This group can be defined as the automorphism group of the functor $\mathrm{Fib}_b,$ which associates to a Galois covering of $\overline{X}$ its fibre over $b$. In addition, if $x_v\in X(F_v)$ for any place $v$ of $F$, we have \it path torsors \rm $\pi_1(\overline{X};b, x_v)$, elements of which are isomorphisms from $\mathrm{Fib}_b$ to $\mathrm{Fib}_{x_v}$. This is clearly a right torsor for $\pi_1(\overline{X},b)$. Furthermore, it carries a natural action of $\Gamma_v$ which is compatible (in a suitable sense) with the action of $\Gamma_v$ on $\pi_1(\overline{X},b)$ that is inherited from the action of $\Gamma_F$ on $\pi_1(\overline{X},b).$ The path torsors $\pi_1(\overline{X};b,x_v)$ thus define elements of the non-Abelian continuous cohomology set $H^1(\Gamma_v, \pi_1(\overline{X},b)).$ 
In this way, we can construct a map
$$j\colon X(\A_F)\to \prod_v H^1(\Gamma_v, \pi_1(\overline{X},b)),$$ which is usually called the \it period map. \rm On the other hand, we also have a natural map
$$H^1(\Gamma_F, \pi_1(\overline{X},b))\to \prod_v H^1(\Gamma_v, \pi_1(\overline{X},b)).$$ One sees easily that any adelic point $P$ of $X$ which comes from a rational point has the property that $j(P)$ lies in the image of this second map. By considering all adelic points of $X$ which share this property, we can define another set which lies between $X(F)$ and $X(\A_F)$. \\
Recently, Kim \cite{K} proposed an iterative construction of a subset $X(\A_F)_\infty$ of the set $X(\A_F)$ which contains the diagonal image of the set of rational points. His construction assumes that $X$ has a rational point and some other conditions (called [Coh1] and [Coh2] in \cite{K}) which we shall state explicitly later. In this paper, we interpret this construction in terms of the descent obstruction introduced by Harari and Skorobogatov \cite{HS}. More precisely, we have the following
\begin{theorem} (Theorem \ref{mainresult})
Let $X$ be a geometrically integral smooth quasi-projective variety over $F$ with a rational point $b$. Assume that $X$ satisfies the conditions [Coh 1] and [Coh 2] from \cite{K}, p. 315 (see also the Appendix of this paper). Let $X(\A_F)^{\mathrm{Nil}}$ be the subset of the set of adelic points of $X$ which survive all torsors $Y\to X$ for finite étale nilpotent group schemes of odd order over $F$ (see Definition \ref{survivesdefinition}). Then
$$X(\A_F)^{\mathrm{Nil}}=X(\A_F)_\infty.$$
\end{theorem}
In fact, Theorem \ref{mainresult} gives a more precise description of the relation between the descent obstruction and Kim's construction; this will become clearer in the following chapters.\\
\\
$\mathbf{Acknowledgement}.$ The author takes pleasure in expressing his deep gratitude to his supervisor, Professor A. Skorobogatov, for suggesting this problem and for his continued support and guidance. He would also like to thank Professor M. Kim for very helpful suggestions, without which the Appendix would not have been written, as well as to Dr. N. Dogra for several enlightening discussions. This work was supported by the Engineering and Physical Sciences Research Council [EP/L015234/1], and the EPSRC Centre for Doctoral Training in Geometry and Number Theory (London School of Geometry and Number Theory), University College London.
\section{Kim's construction}
Recently, Kim \cite{K} defined a filtration
$$X(\A_F)=X(\A_F)_1\supseteq X(\A_F)^2_1\supseteq X(\A_F)_2\supseteq...\supseteq X(\A_F)_n\supseteq X(\A_F)^{n+1}_n\supseteq...$$ If one sets
$$X(\A_F)_{\infty}:=\bigcap_{n=1}^\infty X(\A_F)_n,$$ 
one finds that $$X(F)\subseteq X(\A_F)_\infty.$$
If $X=\Gm,$ the filtration is only non-trivial at the first level because the geometric fundamental group is Abelian, and can be constructed by means of global class field theory.\\
In this subsection, let us review this construction briefly; this will be needed later in the proof. For more details, the reader should consult \cite{K}. \\
More precisely, the sets $X(\A_F)^n_{n-1}$ and $X(\A_F)_n$ are constructed as follows: Let $\pi_1(\overline{X},b)$ be the $\et$ale fundamental group of $\overline{X}$ with base point $b$ and let $\Delta^{[1]}$ be its maximal prime-to-2 quotient. (One has to work with profinite groups which do not have any quotients of even order to deal with certain technical issues at the infinite places of $F$; see \cite{K}, Proposition 2.1 and the comment thereafter.) Define inductively
$$\Delta^{[n+1]}:=\overline{[\Delta^{[1]}, \Delta^{[n]}]}$$ and
$$\Delta_n:=\Delta^{[1]}/\Delta^{[n+1]}.$$ Also define $T_n:=\Delta^{[n]}/\Delta^{[n+1]}$, which leads to a central extension
$$0\to T_n\to \Delta_n\to \Delta_{n-1}\to 0$$ for $n\geq 2.$ The groups $\Delta^{[n]}$, $\Delta_n$ and $T_n$ are all profinite, and $T_n$ is Abelian for all $n\geq 1$.

If $M$ denotes a finite set of odd primes, we denote by $\Delta^M_n$ the maximal pro-$M$ quotient of $\Delta_n$. One can see easily that this is the same as the quotient of $(\Delta^{[1]})^M$ modulo the $n$th element of its lower central series. We denote by $T_n^M$ the graded pieces of the lower central series of $(\Delta^{[1]})^M$.

The groups $\Delta_n$ and $\Delta_n^M$ have the following property: Suppose $G$ is a finite discrete group of nilpotency class $\leq n$. Then any continuous homomorphism $\Delta^{[1]}\to G$ factors through $\Delta_n$. If, moreover, $G$ has order divisible only by primes in $M$, then the morphism $\Delta^{[1]}\to G$ factors through $\Delta_n^M$. 

Let $\Gamma_v$ be a decomposition group at $v$, which we identify with the absolute Galois group of $F_v$. From now on, we shall always assume that the conditions [Coh 1] (that $T^M_n$ is torsion-free for all $M$) and [Coh 2] (that $H^0(\Gamma_v, T^M_n)=0$ for all non-archimedean places $v$ of $F$) are satisfied. These conditions are needed for technical reasons in Kim's construction to prove injectivity of some localization map and for related purposes (see \cite{K}, p. 318f), which is the only reason why we include them here. It is claimed without proof in \cite{K} that these conditions are satisfied if $X$ is a smooth, projective, and geometrically integral curve over $F$. We shall give a proof of this fact in the Appendix.\\
Suppose $S$ is a finite set of places of $F$. For fixed $M$, we call $S$ \it admissible \rm if the action of $\Gamma_F$ on $\Delta^M_n$ factors through $\Gamma_F^S:=\Gal(F_S/F)$, where $F_S$ is the largest extension of $F$ unramified outside $S$. This condition is also required for technical reasons in this construction, in order to apply Poitou-Tate duality. One can show that such an $S$ always exists, but the precise construction of such sets $S$ for given $M$ is rather involved; see \cite{K}, p. 316. In any case, we will not use any facts about these sets other than their existence.\\
Recall that, for each $n$ and $M$, we have a \it period map \rm
$$j_n^M\colon X(\A_F)\to \prod_v\nolimits^{'} H^1(\Gamma_v, \Delta^M_n)$$ given by
$$(x_v)_v\mapsto([\pi_1(\overline{X};b, x_v)\overset{\pi_1(\overline{X},b)}{\times} \Delta^M_n])_v$$ (see \cite{K}, Chapter 3).
Here, the restricted product is taken with respect to the subsets $H^1(\Gamma_v/I_v, (\Delta_n^M)^{I_v})$ of $H^1 (\Gamma_v, \Delta^M_n),$ where $I_v$ is the inertia group at $v$. This makes sense because the element $\pi_1(\overline{X}; b,x_v)$ actually lies in $H^1(\Gamma_v/I_v, (\Delta_n^M)^{I_v})$ for all but finitely many $v$ (\cite{K}, p. 325). It does not matter whether we take the restricted product over all places of $F$ or only the non-Archimedean ones, since the cohomology groups at the infinite places vanish due to the fact that all finite quotients of the groups $\Delta_n^M$ have order prime to 2.  Also recall the localization map
$$H^1(\Gamma_F^S, \Delta_n^M)\to \prod\limits_v \nolimits^{'} H^1(\Gamma_v, \Delta^M_n),$$ which is injective by \cite{K}, p. 319.

We define $X(\A_F)_1:=X(\A_F).$ First, one constructs a map
$$prec_1\colon \varprojlim_M \prod\limits_v \nolimits^{'} H^1(\Gamma_v, \Delta^M_1)\to H^1(\Gamma_F, D(\Delta_1))^{\vee}$$ and defines
$$X(\A_F)^2_1:=(prec_1\circ j_1)^{-1}(0),$$
where $j_n:=\varprojlim j_n^M.$ Here, $D(-)$ refers to the continuous Galois dual $\Hom_{\mathrm{cont}}(-, \mu_{\infty})$ and $-^\vee$ refers to the Pontriagin dual $\Hom_{\mathrm{cont}}(-, \Q/\Z).$ The kernel of the map $prec_1$ is precisely $\varprojlim_M\varinjlim_S H^1(\Gamma_F^S, \Delta^M_1)$ by Poitou-Tate duality (see \cite{K}, Proposition 2.1 and \cite{N}, Theorem 8.6.7). Now, we have a connecting map
$$prec^2_1:=\delta^g_1\colon \varprojlim_M\varinjlim_S H^1(\Gamma_F^S, \Delta^M_1)\to \varprojlim_M\varinjlim_S H^2(\Gamma_F^S, T^M_2)$$ coming from the central extension above. The set $X(\A_F)_2$ is now defined to be the set of all elements of $X(\A_F)^2_1$ on which the composition $prec^2_1\circ j_1\colon X(\A_F)^2_1\to\varprojlim_M\varinjlim_S H^2(\Gamma_F^S, T^M_2)$ vanishes. Now consider the projection map
$$p_1\colon \prod_v^{'} H^1(\Gamma_v, \Delta_2^M)\to \prod_v^{'} H^1(\Gamma_v, \Delta^M_1).$$ The kernel of the connecting map
$$\varinjlim_S H^1(\Gamma_F^S, \Delta^M_1)\to \varinjlim_S H^2(\Gamma_F^S, T^M_2)$$
can be viewed as a subset of the target of $p_1$, and we define $W(\Delta_2^M)$ to be the inverse image of this kernel under $p_1$.
Now one defines a map
$$prec_2\colon \varprojlim_M W(\Delta_2^M)\to H^1(\Gamma_F, D(T_2))^\vee.$$
Since the image of $X(\A_F)_2$ under the period map $j_2$ is contained in $\varprojlim_M W(\Delta_2^M)$, we can now define
$$X(\A_F)^3_2:=(prec_2\circ j_2)^{-1}(0).$$
From here, one proceeds by induction: Define
$$prec_{n-1}^{n}:=\delta_n^g\colon \varprojlim_M\varinjlim_S H^1(\Gamma_F^S, \Delta_{n-1}^M)\to \varprojlim_M\varinjlim_SH^2(\Gamma_F^S, T^M_{n}),$$ and let $X(\A_F)_n$ be the set of elements of $X(\A_F)_{n-1}^n$ on which the composition
$$X(\A_F)_{n-1}^n\to \varprojlim_M\varinjlim_S H^1(\Gamma_F^S, \Delta_{n-1}^M)\to \varprojlim_M\varinjlim_SH^2(\Gamma_F^S, T^M_{n})$$ vanishes. This makes sense by construction of $X(\A_F)^n_{n-1}.$ Now one defines $W(\Delta_n^M)$ in a way entirely analogous to the construction of $W(\Delta_2^M)$ above, constructs a map
$$prec_{n} \colon \varprojlim_M W(\Delta_{n}^M)\to H^1(\Gamma_F, D(T_{n}))^\vee,$$ and defines
$$X(\A_F)^{n+1}_n:=(prec_n\circ j_n)^{-1}(0).$$
The reader will have noticed several double limits of the form
$$\varprojlim_M\varinjlim_S H^1(\Gamma^S_F, \Delta^M_n).$$ They are to be interpreted in the following way: For fixed $M$, one considers the set of all admissible sets $S$ of places of $F$, ordered by inclusion, and calculates the limit $\varinjlim_S H^1(\Gamma_F^S, \Delta_n^M).$ These sets naturally form a projective system, indexed by $M$ (with the set of all $M$s also ordered by inclusion), and one can now make sense of the limit $\varprojlim_M(\varinjlim_S H^1(\Gamma^S_F, \Delta^M_n)).$

\section{The Galois action on the geometric $\mathbf{\acute{e}}$tale fundamental group}
Consider the homotopy exact sequence \label{galoisactionsection}
$$0\to \pi_1(\overline{X},b)\to \pi_1(X,b)\to \Gal(\overline{F}/F)\to 0.$$ Then the rational point $b\colon \Spec F\to X$ gives rise to a section $b_\ast\colon \Gal(\overline{F}/F)\to \pi_1(X,b)$, so $\Gal(\overline{F}/F)$ operates on $\pi_1(\overline{X},b)$ by conjugation.\\
\\
Let us now recall the profinite structure of $\pi_1(\overline{X},b).$ This is necessary to set up some notation which will be very important later. Let $\{Y^{\alpha}\}$ be the set of all Galois covers $Y^{\alpha}\to \overline{X}.$ Note that such a Galois cover $Y^{\alpha}$ is nothing but a connected right torsor for a finite $\et$ale group scheme $G_\alpha$ over $\overline{F}.$ We define a partial ordering on the index set by putting $\beta\leq\alpha$ if there is an $\overline{X}$-morphism $\phi_{\alpha\beta}\colon Y^\alpha\to Y^\beta$. Furthermore, we assume that a system of such maps and geometric points $p_\alpha\in Y_b^\alpha(\overline{F})$, satisfying $\phi_{\alpha\beta}(p_\alpha)=p_\beta$, has been chosen. Note that for $\beta \leq \alpha$, there is at most one map satisfying this condition. Such a system together with a choice of geometric points can always be constructed, albeit in a non-canonical way (see \cite{Sz}, Chapter 5.4). For each $\alpha$, we have a canonical isomorphism
$$\mathrm{Aut}_{\overline{X}}(Y^\alpha)^{\mathrm{op}}\cong G_\alpha(\overline{F})$$ (we have to consider the opposite group here because $G_\alpha$ acts from the right on $Y^\alpha$, whereas, by convention, an automorphism group always acts from the left). For $\beta\leq\alpha$ the maps $p_{\alpha\beta}$ induce maps
$$\mathrm{Aut}_{\overline{X}}(Y^\alpha)\to \mathrm{Aut}_{\overline{X}}(Y^\beta),$$ which, in turn, give rise to an isomorphism
$$\pi_1(\overline{X},b)\cong \varprojlim_\alpha \mathrm{Aut}_{\overline{X}}(Y^\alpha)^{\mathrm{op}}\cong \varprojlim_\alpha G_\alpha(\overline{F}).$$
(For more details about this construction, see \cite{Sz}, Chapter 5.4). This isomorphism explains why the system $Y^{\alpha}$ is referred to as a \it pro-representing system. \rm\\
Now suppose that, for some $\alpha$, $Y^\alpha$ arises from a geometrically connected torsor $Y\to X$ for some finite $\et$ale $F$-group scheme $G$. Then $G(\overline{F})$ automatically carries an action of $\Gal(\overline{F}/F),$
and we obtain an induced homomorphism of topological groups
$$\pi\colon \pi_1(\overline{X},b)\to G(\overline{F}).$$ In general, this homomorphism will not be equivariant, since the Galois action on the finite group $G(\overline{F})$ is not determined uniquely by the base change of the torsor $Y\to X$ to $\overline{X}$. As it turns out, however, a small adjustment will suffice to deal with this problem. \begin{lemma}
Let $Y$ be a geometrically connected right torsor for $G$. Let $\tau\colon \Gamma_F\to G$ be the cocycle of the class of the fibre $Y_b$ in $H^1(\Gamma_F, G)$ specified by
$${}^\sigma y=y\cdot \tau(\sigma),$$ where $y\in Y_b(\overline{F})$ is the distinguished geometric point coming from the pro-representing system specified above. Then the map \label{calculationlemma}
$$\pi\colon \pi_1(\overline{X},b)\to G^\tau$$ is $\Gamma_F$-equivariant.
\end{lemma}
\begin{proof}
The map $\pi$ is specified by the requirement that
$$\gamma(y)=y\cdot \pi(\gamma)$$ for all $\gamma\in\pi_1(\overline{X},b).$ We compute
\begin{align*}
y\cdot\tau(\sigma)\ {}^{\sigma}\pi(\gamma)\tau(\sigma)^{-1}&=\ {}^\sigma(y\cdot \pi(\gamma))\cdot\tau(\sigma)^{-1}\\
&=\ {}^\sigma\gamma(y)\cdot\tau(\sigma)^{-1}\\
&=y\cdot \pi(\ {}^\sigma\gamma),
\end{align*}
where the last equality is due to the fact that $\gamma$ is an automorphism of the fibre functor, and hence commutes with the right action of $G$. We also used the cocycle condition several times. The claim follows from this calculation.
\end{proof}

We are now in a position to describe the Galois action on finite quotients of $\pi_1(\overline{X},b),$ provided they arise as in Lemma \ref{calculationlemma}. The remainder of this section is dedicated to showing that understanding these finite quotients suffices to understand the Galois action on $\pi_1(\overline{X},b).$ \label{cofinalsystemdefinition}
\begin{definition} \rm(Compare Stoll \cite{S}, p. 366) \it Let $\mathcal{P}$ be a property of finite groups. A \rm cofinal family \it of torsors with property $\mathcal{P}$ is a family $\{Y^i\to X\}$ of geometrically connected torsors for finite $\acute{e}$tale  group schemes $G_i$ which have $\mathcal{P}$ such that for all connected torsors $Z\to \overline{X}$ for a finite group that has $\mathcal{P}$, there is a map of torsors $\overline{Y}^i\to Z$ for some $i$. 
\end{definition}
Recall that the finite $\et$ale $F$-group scheme $G$ has property $\mathcal{P}$ if the finite group $G(\overline{F})$ does. The properties $\mathcal{P}$ that will be important for us will always be such that if a finite group $G$ has $\mathcal{P}$ and $H\subseteq G$ is a subgroup, then $H$ has $\mathcal{P}$, and if $G$ and $H$ are finite groups both of which have $\mathcal{P}$, then $G\times H$ has $\mathcal{P}.$
\begin{lemma}
Suppose that $\mathcal{P}$ is such that\\
(i) whenever $K/F$ is a finite extension and $G$ a finite $\acute{e}$tale group scheme over $K$ which has $\mathcal{P}$, then $\Res_{K/F} G$ has $\mathcal{P}$, and\\
(ii) if $G$ has $\mathcal{P}$ then so does every subgroup scheme of $G$.\\ 
\label{cofinalsystemsexistencelemma}Then there is a cofinal family of torsors with property $\mathcal{P}$. In particular, there exists a cofinal family of torsors under group schemes of odd order and of nilpotency class $\leq n$, for all $n$. 
\end{lemma}
\begin{proof}
Let $Z\to \overline{X}$ be a connected right torsor for some finite group $G$ which has $\mathcal{P}$. Then $Z$ is, in fact, defined over some finite extension $K$ of $F$, i.e. $Z$ arises as the extension of scalars of some geometrically connected right torsor $Z_K\to X\times_F\Spec K.$ Following the proof of Stoll \cite{S}, Lemma 5.7, we obtain a geometrically connected right torsor
$$\Res_{K/F} Z_K\to \Res_{K/F} (X\times_F\Spec K)$$ for the group scheme $\Res_{K/F} G.$ By assumption, this still has $\mathcal{P}.$ Pulling this back to $X$ via the canonical morphism $X\to \Res_{K/F} (X\times_F\Spec K)$, we obtain a torsor $Y\to X$ for $\Res_{F/K} G.$ From this, we obtain the diagram
$$\begin{CD}
Y\times_F\Spec K@>>>( \Res_{K/F} Z_K)\times_F\Spec K@>>>Z_K\\
@VVV@VVV@VVV\\
X\times_F\Spec K @>>>(\Res_{K/F} (X\times_F\Spec K))\times_F\Spec K@>>> X\times_F\Spec K,
\end{CD}$$
the two right horizontal arrows being the canonical maps, which come from the adjoint property of extension of scalars and Weil restriction. In particular, $Y\times_F\Spec \overline{F}$ maps to $Z.$
After possibly twisting $Y$, we may assume that $Y$ lifts the rational point $b\in X(F).$ Hence, a connected component $C$ of this twist will have a rational point, which will therefore be geometrically connected. This connected component is a torsor for its stabilizer, which still has property $\mathcal{P}$ by assumption. Since $\overline{C}$ visibly maps to $Z$, the claim follows.
\end{proof}\\
$\mathbf{Remark.}$ There is a small error in \cite{S}, p. 364, where it is claimed that an irreducible component of a torsor is itself a torsor for its stabilizer. This is inaccurate, as the group scheme $\boldsymbol{\mu}_3\to \Spec \Q$ shows (it is a torsors for itself, and it has an irreducible component isomorphic to $\Spec \Q[X]/\langle X^2+X+1\rangle$ whose stabilzer is trivial). However, the following is true: If $Y\to X$ is a torsor for a finite étale group scheme $G$ over $F$ and $Y_0\subseteq Y$ is an irreducible component which is \it geometrically \rm irreducible, then the stabilizer of $\overline{Y}_0$ is defined over $F$, and $Y_0$ is a torsor for  its stabilizer. The same remark applies to our proof of Lemma \ref{surviveslemmaII}.\\
\\
We have already observed that the topological group $\pi_1(\overline{X},b)$ is isomorphic to the inverse limit of the (opposite groups of) the automorphism groups of the $Y^\alpha$. The previous lemma shows that we can also understand the action of $\Gal(\overline{F}/F)$ on $\pi_1(\overline{X},b)$ in profinite terms: Suppose $Y\to X$ and $Z\to X$ are geometrically connected right torsors for some finite $\et$ale group schemes $G$ and $H$ over $F$. They induce Galois covers $\overline{Y}\to \overline{X}$ and $\overline{Z}\to\overline{X}$ with Galois groups $G(\overline{F})^{\mathrm{op}}$ and $H(\overline{F})^{\mathrm{op}}$, respectively. Suppose that $\overline{Z}\leq \overline{Y}$ with respect to the ordering on the set of all Galois covers of $\overline{X}$ introduced before. If we now define $\tau:=[Y_b]\in H^1(\Gamma_F,G)$ and $\mu:=[Z_b]\in H^1(H,G)$, we obtain a $\Gamma_F$-equivariant map
$G^\tau\to H^\mu.$ The inverse system that arises this way is cofinal with the inverse system that arises from the set of all Galois covers of $\overline{X}$. Hence we obtain

\begin{corollary}
There is a $\Gamma_F$-equivariant isomorphism
$$\pi_1(\overline{X},b)\cong \varprojlim_i G_i^{\tau_i},$$ where $\{Y^i\to X\}$ is a cofinal family of torsors for finite étale group schemes over $F$ ($G_i$ being the structure group of $Y^i$) and the maps between them are induced from the pro-representing system specified in section \ref{galoisactionsection}. Analogously, there is a $\Gamma_F$-equivariant isomorphism
$$\Delta^M_n\cong \varprojlim_i G_i^{\tau_i},$$ where $\{Y^i\to X\}$ is a cofinal family of torsors for finite étale group schemes over $F$ of nilpotency class $\leq n$ and of order divisible only by primes in $M$. 
\end{corollary}
$\mathbf{Example.}$ Let $E$ be an elliptic curve over $F$. From the self-duality of elliptic curves it follows easily that the system of Galois coverings of $\overline{E}:=E\times_F\Spec\overline{F}$ defined by
$$[n]\colon \overline{E}\to \overline{E}$$ for $n\in \N$ is cofinal for the system of all Galois coverings of $\overline{E}$. Now, the map $[n]\colon E\to E$ over $F$ defines a (right) torsor for the finite $\et$ale group scheme $E[n]$ over $F$. Therefore, the system $[n]\colon E\to E$, where $n\in \N$, is a cofinal family of torsors in the sense of Definition 1.3. Since everything is Abelian in this example, no twisting is necessary, and we obtain the familiar isomorphism
$$\pi_1(\overline{E},0)\cong \prod_{p\, \text{prime}} T_p(E) =\varprojlim E[n],$$ where the set $\N$ of natural numbers is ordered by divisibility.

\section{Reciprocity and the descent obstruction}
Let $n\in \N$ and suppose that $M$ is a finite set of odd primes. 
Then let $\Delta_n$ be the quotient of $\Delta^{[1]}$ (which is, by definition, the maximal prime-to-2-quotient of $\pi_1(\overline{X}, b))$ modulo the $n$-th object of the (closed) lower central series of $\pi_1(\overline{X},b)$ and let $\Delta_n^M$ be the maximal pro-$M$ quotient of $\Delta_n$. 
Recall the map
$$j_n^M\colon X(\A_F)\to \prod\nolimits^{'} H^1(\Gamma_v, \Delta^M_n)$$
given by 
$$(x_v)_v\mapsto([\pi_1(\overline{X};b, x_v)\overset{\pi_1(\overline{X},b)}{\times} \Delta^M_n])_v$$ (see \cite{K}, Chapter 3). Here, we are using continuous non-Abelian cohomology. The next proposition shows how the non-Abelian reciprocity law fits into our framework:
\begin{proposition}
The diagram \label{cartesianproposition}
$$\begin{CD}
X(\A_F)^{n+1}_n@>>>X(\A_F)\\
@VVV @VV{j_n}V\\
\varprojlim_M \varinjlim_S H^1(\Gamma_F^S,\Delta^M_n)@>>> \varprojlim_M \prod_v^{'}H^1(\Gamma_v, \Delta^M_n)
\end{CD}$$
is Cartesian.
\end{proposition}
\begin{proof}
We shall prove this by induction on $n$. If $n=1$, we have $X(\A_F)_n=X(\A_F)$ by definition. The set $X(\A_F)^2_1$ is defined to be the set of elements of $X(\A_F)=X(\A_F)_1$ on which the composition
$$X(\A_F)\overset{j_1}{\longrightarrow} \varprojlim_M \prod_v^{'}H^1(\Gamma_v, \Delta^M_1)\overset{prec_1}{\longrightarrow} H^1(\Gamma_F, D(\Delta_1))^\vee$$ vanishes. But the kernel of $prec_1$ is exactly the image of $\varprojlim_M \varinjlim_S H^1(\Gamma_F^S,\Delta^M_1)$ by Poitou-Tate duality (see \cite{K}, p. 322), so the claim follows in this case. Now assume that we know the statement for $n-1$. Consider the commutative diagram

$$\begin{CD}
P@>>>X(\A_F)\\
@VVV@VVV\\
\varprojlim_M\varinjlim_S H^1(\Gamma_F^S, \Delta_n^M)@>>>\varprojlim_M\prod_v^{'} &H^1(\Gamma_v, \Delta_n^M)\\
@VVV@VVV\\
\varprojlim_M\varinjlim_S H^1(\Gamma_F^S, \Delta_{n-1}^M)@>>>\varprojlim_M\prod_v^{'} &H^1(\Gamma_v, \Delta_{n-1}^M)
\end{CD},$$
where $P$ is the pullback of $X(\A_F)$ with respect to the middle horizontal arrow. By the induction hypothesis, we know that the pullback of $X(\A_F)$ along the bottom horizontal arrow is equal to $X(\A_F)^n_{n-1}.$ This clearly contains $P$. Consider the composition
$$P\to \varprojlim_M\varinjlim_S H^1(\Gamma_F^S, \Delta_{n-1}^M)\overset{prec^{n}_{n-1}}{\longrightarrow}  \varprojlim_M\varinjlim_S H^2(\Gamma_F^S, T_{n}^M).$$ This composition vanishes because the first map factors through $\varprojlim_M\varinjlim_S H^1(\Gamma_F^S, \Delta_{n}^M).$ Hence we find $P\subseteq X(\A_F)_n.$ However, by \cite{K}, Proposition 2.4, we have $\ker prec_{n}=\varprojlim_M\varinjlim_S H^1(\Gamma_F^S, \Delta_n^M)$.
By definition, $X(\A_F)_n^{n+1}$ is the set of elements of $X(\A_F)_n$ on which $prec_n\circ j_n$ vanishes, so we find $P=X(\A_F)_n^{n+1}.$ This implies the result.
\end{proof}

Now let $v$ be a place of $F$ and let $x_v\in X(F_v).$ Then the local Galois group $\Gal(\overline{F}_v/F_v)=:\Gamma_v\subseteq \Gamma_F$ acts on the right path torsor $\pi_1(\overline{X},b,x_v)$ in a fashion compatible with the action on $\pi_1(\overline{X},b).$
\begin{proposition}
Assume that $Y\to X$ is a geometrically connected right torsor for some finite $\acute{e}$tale $F$-group scheme $G$. Let $\tau\colon \Gamma_v\to G(\overline{F})$ be the cocycle given by ${}^\sigma y=y\cdot \tau(\sigma)$, where $y$ is the distinguished geometric point of $Y_b$ coming from the pro-representing system we specified in Section \ref{galoisactionsection}.  Its class in $H^1(\Gamma_v,G)$ is $[Y_b]$. Then the image of $[\pi_1(\overline{X};b,x_v)]$ in $H^1(\Gamma_v, G^\tau)$ under the map \label{imagesproposition}
$$H^1(\Gamma_v, \pi_1(\overline{X},b))\to H^1(\Gamma_v, G^\tau)$$ is $[Y_{x_v}\overset{G}{\times}Y_b^{-1}].$
\end{proposition}
\begin{proof}
Let $y\in Y_b(\overline{F})$ be the distinguished geometric point. Define a map of sets
$$\pi_1(\overline{X};b,x_v)\overset{\pi_1(\overline{X},b)}{\times} G^\tau\to Y_{x_v}\overset{G}{\times} Y_b^{-1}$$
by
$$(\gamma,g)\mapsto (\gamma(y\cdot g),y).$$
One checks directly that this is well-defined and commutes with the action of $G^\tau$, using the functoriality of the fibre functor. Galois equivariance also follows directly. Since $\Gamma_v$-equivariant maps of torsors are isomorphisms, the claim follows.
\end{proof}\\
The remainder of this paper will consist in proving that Kim's construction can be understood in terms of the finite descent obstruction. 
The following is a slight modification of Definition 5.3.1 form \cite{Sk}. We shall use the notion of \it twists of torsors by cocycles; \rm see \cite{Sk}, p. 21 for an introduction to this operation.
\begin{definition} 
We say that an adelic point $(x_v)_v$ \rm survives \it a torsor $\pi\colon Y\to X$ for some finite $\acute{e}$tale $F$-group scheme $G$ if \label{survivesdefinition}
$$(x_v)_v\in \bigcup_{\sigma\in H^1(\Gamma_F,G)} \pi^{\sigma}(Y^\sigma(\A_F)).$$
Moreover, we define the set $X(\A_F)^{\mathrm{Nil}_n}$ to be the set of adelic points of $X$ which survive all finite $\acute{e}$tale group schemes $G$ of odd order with $G(\overline{F})$ of nilpotency class $\leq n$, and define
$$X(\A_F)^{\mathrm{Nil}}:=\bigcap_{n\in\N} X(\A_F)^{\mathrm{Nil}_n}.$$ 
\end{definition}
The restriction that the order of $G$ be odd has been added in the definition above only because it will make the statement of some theorems later more convenient.

\begin{proposition} Let $(x_v)_v\in X(\A_F)^{n+1}_n.$ Let $\pi\colon Y\to X$ be a finite geometrically connected $\acute{e}$tale covering which is a torsor for some finite $\acute{e}$tale group scheme $G$, which we assume to be of nilpotency class $\leq n$ and of order divisible only by primes in $M$, for some finite set $M$ of odd prime numbers. Then there is an element $\sigma\in H^1(\Gamma_F, G)$  such that
$$(x_v)_v\in \pi^\sigma(Y^\sigma(\A_F)).$$
\end{proposition}
\begin{proof}
Because $(x_v)_v\in X(\A_F)^{n+1}_n$, the image of $(x_v)_v$ in $\prod^{'}_vH^1(\Gamma_v, \Delta^M_n)$ actually lies in the image of
$$\varinjlim_S H^1(\Gamma^S_F, \Delta^M_n)\to \prod^{'}_vH^1(\Gamma_v, \Delta^M_n)$$ by Proposition \ref{cartesianproposition}.
The maps $H^1(\Gamma_F^S, \Delta^M_n)\to H^1(\Gamma_F, \Delta^M_n)$ are compatible, so they give rise to a map
$$\varinjlim_S H^1(\Gamma_F^S, \Delta^M_n)\to H^1(\Gamma_F, \Delta^M_n).$$ Furthermore, since $b\in X(F)$, there is $\tau\in H^1(\Gamma_F,G)$ such that $b\in \pi^\tau(Y^\tau(\A_F))$. Now, the twisted inner form $G^\tau$ of $G$ can be seen as a group isomorphic to $G$ with a $\Gamma_F$-action on it. The finite $\et$ale cover $Y^\tau\to X$ gives rise to a Galois equivariant map
$\Delta^M_n\to G^\tau,$ and hence maps
$H^1(\Gamma_F, \Delta^M_n)\to H^1(\Gamma_F, G^\tau)$ and $H^1(\Gamma_v, \Delta^M_n)\to H^1(\Gamma_v, G^\tau).$ By Proposition \ref{imagesproposition}, we already know that the image of $[\pi_1(\overline{X}; b, x_v)\overset{\pi_1(\overline{X},b)}{\times} \Delta^M_n]$ under this last map is $[Y^\tau_{x_v}]$.
Because the diagram
$$\begin{CD}
H^1(\Gamma_F, \Delta^M_n)@>>> H^1(\Gamma_v, \Delta^M_n)\\
@VVV@VVV\\
H^1(\Gamma_F, G^\tau) @>>> H^1(\Gamma_v, G^\tau)
\end{CD}$$
is commutative for all places $v$ of $F$, we can define $\sigma$ to be the image of $(x_v)_v$ in $H^1(\Gamma_F, G^\tau).$ Indeed, the image of $\sigma$ in $H^1(\Gamma_v, G^\tau)$ is $[Y^\tau_{x_v}]$, so $(x_v)_v$ lifts to $(Y^\tau)^\sigma.$ Now we can use the canonical bijection between $H^1(\Gamma_F, G^\tau)$ and $H^1(\Gamma_F, G)$ to deduce the result.
\end{proof} \\
For the next lemma, we need the following\\
\\
$\mathbf{Claim}.$ Let $i=1,2$ and let $Y_i\to X$ be torsors for finite $\et$ale group schemes $G_i$. Suppose that $Y_1$ is geometrically connected, and assume that there is a map of torsors $\overline{Y}_1\to \overline{Y}_2$ over $\overline{X}$. Then there is a cocycle $\sigma\colon \Gamma_F\to G_2(\overline{F})$ such that the map $\overline{Y}_1\to \overline{Y}_2$ descends to a map $Y_1\to Y_2^\sigma$ over $X$.\\
\\
\begin{proof}
The proof of Stoll \cite{S}, Lemma 5.6, can be taken without modification.
\end{proof}

\begin{lemma} 
Let $Y^i\to X$ be a cofinal family of torsors for finite $\acute{e}$tale group schemes of odd order and nilpotency class $\leq n$, which exists by Lemma \ref{cofinalsystemsexistencelemma}. Let $G_i$ be the structure group scheme of $Y^i\to X.$ Then an adelic point of $X$ is contained in $X(\A_F)^{\mathrm{Nil}_n}$ if and only if it survives all $Y^i$. \label{surviveslemmaII}
\end{lemma}
\begin{proof}
We follow the proof of Stoll \cite{S}, Lemma 5.7. One direction is obvious, so let us consider the other one. Suppose $Y\to X$ is any right torsor for some finite $\et$ale $F$-group scheme $G$ satisfying the hypotheses of the Lemma. Then $\overline{Y}\to \overline{X}$ is a (not necessarily connected) right torsor for the finite group $G(\overline{F}).$ Let $Y_0$ be any connected component of $\overline{Y}$, and let $G_0$ be its stabilizer. Then $Y_0$ is a connected torsor for $G_0$, and $G_0$ still has odd order and nilpotency class $\leq n.$ Hence, by assumption, there is a member $Y^i$ of our cofinal family such that we have a map $\overline{Y}^i\to Y_0\to \overline{Y}$ over $\overline{X}$. By the claim above the lemma, there is a twist $Y^\sigma$ of $Y$ such that we obtain a morphism $Y^i\to Y^\sigma$. 
Now consider any adelic point $P$ of $X$. By assumption, $P$ will survive $Y^i$. Hence it will also survive $Y$.
\end{proof}

\begin{corollary}
We have $X(\A_F)^{n+1}_n\subseteq X(\A_F)^{\mathrm{Nil_n}}.$ In particular, \label{firstinclusioncorollary}
$$X(\A_F)_\infty\subseteq X(\A_F)^{\mathrm{Nil}}.$$
\end{corollary}

This concludes the proof of one direction of our main result, and we shall now consider the opposite direction. Suppose that we have an adelic point $(x_v)_v\in X(\A_F)$ with the property that for all torsors $\pi\colon Y\to X$ for some finite $\et$ale group scheme $G$ (of nilpotency class $\leq n$ and of order divisible only by primes in $M$) there exists an element $\sigma\in H^1(\Gamma_F,G)$ such that $(x_v)_v\in \pi^\sigma(Y^\sigma(\A_F)).$ In order to show that $(x_v)_v\in X(\A_F)^{n+1}_n$, it suffices to show that the family of local classes $[\pi_1(\overline{X};b,x_v)\overset{\pi_1(\overline{X},b)}{\times} \Delta^M_n]\in \prod_vH^1(\Gamma_v, \Delta^M_n)$ comes from a global class, i.e. from an element of $H^1(\Gamma_F, \Delta^M_n):$ 

\begin{lemma}
Let $(x_v)_v \in X(\A_F)$. Let $Y\to X$ be a geometrically connected torsor for a finite $\acute{e}$tale group scheme $G$. Let $\tau:=[Y_b]\in H^1(\Gamma_F, G)$ and suppose that there is a class $\mu\in H^1(\Gamma_F, G)$ such that $(x_v)_v$ lifts to $Y^\mu$. Then the image of $([\pi_1(\overline{X};b, x_v)])_v$ in $\prod_v H^1(\Gamma_v, G^\tau)$ under the map
$$\prod_vH^1(\Gamma_v, \pi_1(\overline{X},b))\to \prod_v H^1(\Gamma_v, G^\tau)$$ lies in the image of $H^1(\Gamma_F, G^\tau)\to \prod_vH^1(\Gamma_v, G^\tau)$. \label{liesintheimagelemma}
\end{lemma}
\begin{proof}
Let $v$ be a place of $F$. By Proposition \ref{imagesproposition}, we already know that the image of $[\pi_1(\overline{X};b, x_v)]$ in $H^1(\Gamma_v, G^\tau)$ is $[Y_{x_v}\overset{G}{\times} Y_b^{-1}].$ But this is the same as $[Y^\mu_{x_v}\overset{G^\mu}{\times} (Y^\mu_b)^{-1}].$ By assumption, $Y^\mu_{x_v}$ has $F_v$-points, so the class
$$[G^\mu\overset{G^\mu}{\times} (Y_b^\mu)^{-1}]\in H^1(\Gamma_F, G^\tau)$$ will do. Note that $\mu$ does not depend on the place $v$. This implies the claim.
\end{proof}
\begin{lemma}\rm (Minkowski-Hermite) \it
Let $G$ be a finite discrete (not necessarily Abelian) group with a continuous action of $\Gamma_F$. Then the localization map \label{finitefibreslemma}
$$H^1(\Gamma_F, G)\to \prod_v H^1(\Gamma_v, G)$$ has finite fibres.
\end{lemma}
\begin{proof} See Stix \cite{St}, Lemma 146.
\end{proof}
\begin{proposition}
Let $n\in \N$. Suppose that $(x_v)_v\in X(\A_F)_n$ survives all torsors $Y\to X$ for finite $\acute{e}$tale group schemes $G$ of nilpotency class $\leq n$ and of order divisible by primes in $M$. Then $j_n^M((x_v)_v)$ lies in the image of $H^1(\Gamma_F,\Delta_n^M)\to \prod_v H^1(\Gamma_v, \Delta^M_n).$ \label{localglobalproposition1}
\end{proposition}
\begin{proof}
Let $(Y^i,G_i)$ be a cofinal family of such torsors, which exists by Lemma \ref{cofinalsystemsexistencelemma}. Define $\tau_i:=[Y^i_b]\in H^1(\Gamma_F, G_i).$ Then we obtain an isomorphism
$\Delta_n^M\cong \varprojlim G_i^{\tau_i}$ and hence an isomorphism
$$H^1(\Gamma_v, \Delta_n^M)\to \varprojlim H^1(\Gamma_v, G_i^{\tau_i})$$ for each place $v$ of $F$. Then the image of $[\pi_1(\overline{X};b,x_v)]$ in $H^1(\Gamma_v, \Delta^M_n)$ corresponds to the projective system $[Y^i_{x_v}\overset{G_i}{\times} (Y_b^i)^{-1}]$. For each $i$, the set $S_i$ of preimages of $([Y^i_{x_v}\overset{G_i}{\times} (Y_b^i)^{-1}])_v\in \prod_v H^1(\Gamma_v, G_i^{\tau_i})$ in $H^1(\Gamma_F, G_i^{\tau_i})$ is non-empty by Lemma \ref{liesintheimagelemma}, and it is finite by Lemma \ref{finitefibreslemma}. Furthermore, the maps $H^1(\Gamma_F, G_j^{\tau_j})\to H^1(\Gamma_F, G_k^{\tau_k})$ for $\overline{Y}^k\leq \overline{Y}^j$ descend to maps $S_j\to S_k$. It follows that the limit $\varprojlim S_i$ is non-empty. Any element of this limit can be viewed as an element of $H^1(\Gamma_F, \Delta_n^M)$. Hence the result follows.
\end{proof}
\begin{lemma}
Let $M$ be a finite set of odd prime numbers and let $S$ be an admissible finite set of places of $F$ (i.e., the action of $\Gamma_F$ on $\Delta^M_n$ factors through $\Gamma_F^S$). Let $\alpha\in H^1(\Gamma_F, \Delta^M_n)$ be such that its image in $H^1(\Gamma_v, \Delta^M_n)$ comes from $H^1(\Gamma_v/I_v, (\Delta^M_n)^{I_v})$ for all places $v\not\in S.$ Then $\alpha$ lies in the image of $H^1(\Gamma_F^S, \Delta^M_n)\to H^1(\Gamma_F, \Delta^M_n).$ \label{localgloballemma}
\end{lemma}
\begin{proof}
Pick a cocycle representing the class $\alpha\in H^1(\Gamma_F, \Delta^M_n)$, and denote it by $\alpha\colon \Gamma_F\to \Delta^M_n.$ By our assumption on $\alpha$, we have $\alpha(I_v)=0$ for all $v\not\in S$. But since the group $\ker(\Gamma_F\to \Gamma_F^S)$ is the closed normal subgroup generated by all $I_v$ for $v\not\in S$ (this follows, for example, from \cite{Na}, Chapter 6.1, Corollary 2, p. 265, and a limit argument) and the cocycles in question are continuous, this implies that $\alpha(\ker(\Gamma_F\to \Gamma_F^S))=0.$ Using that $\ker(\Gamma_F\to \Gamma^S_F)$ acts trivially on $\Delta^M_n$ since $S$ is admissible, we find that $\alpha$ defines a class in $H^1(\Gamma_F^S, \Delta^M_n).$ This concludes the proof.
\end{proof}\\
What follows now is our main result:
\begin{theorem}
Let $X$ be a geometrically integral smooth quasi-projective variety over $F$ which admits a rational point. Assume that $X$ satisfies the conditions [Coh 1] and [Coh 2] from \cite{K}, p. 315 (see also the Appendix of this paper). Then we have \label{mainresult} $X(\A_F)^{\mathrm{Nil_n}}= X(\A_F)^{n+1}_n$ for all $n\in \N$. In particular,
$$X(\A_F)^{\mathrm{Nil}}=X(\A_F)_\infty.$$
\end{theorem}
\begin{proof}
By Corollary \ref{firstinclusioncorollary}, all we have to show is the inclusion "$\subseteq$". Suppose $(x_v)_v \in X(\A_F)^{\mathrm{Nil}_n}.$ By Proposition \ref{localglobalproposition1} we already know that $j_n^M((x_v)_v)$ lies in the image of 
$$H^1(\Gamma_F, \Delta^M_n)\to \prod_v H^1(\Gamma_v, \Delta_n^M)$$ for all finite sets $M$ of odd prime numbers. However, we know that $j_n^M((x_v)_v)$ lies in the restricted product 
$\prod^{'} H^1(\Gamma_v, \Delta_n^M).$ By \cite{K}, Chapter 2, we can find an admissible set $S$ of places of $F$ which contains all places $v$ such that the image of $j_n^M((x_v)_v)$ in $H^1(\Gamma_v, \Delta^M_n)$ does \it not \rm come from $H^1(\Gamma_v/I_v, (\Delta^M_n)^{I_v})$. From Lemma \ref{localgloballemma} it now follows that $j_n^M((x_v)_v)$ lies in the image of
$$\varinjlim H^1(\Gamma_F^S, \Delta_n^M)\to \prod\nolimits^{'} H^1(\Gamma_v, \Delta_n^M).$$

Now suppose $M\subseteq N$ are two finite sets of odd prime numbers. Consider the commutative diagram
$$\begin{CD}
\varinjlim H^1(\Gamma^S_F, \Delta_n^N)@>>>\prod^{'}_v H^1(\Gamma_v, \Delta^N_n)\\
@VVV@VVV\\
\varinjlim H^1(\Gamma^S_F, \Delta_n^M)@>>>\prod^{'}_v H^1(\Gamma_v, \Delta^M_n).
\end{CD}$$
The horizontal arrows are injective because the localization maps $$H^1(\Gamma^S_F, \Delta_n^M) \to \prod^{'}_v H^1(\Gamma_v, \Delta^M_n)$$ are injective (\cite{K}, p. 319). Hence, the preimage of $j_n^N((x_v)_v))$ in $\varinjlim _S H^1(\Gamma^S_F, \Delta_n^N)$ must be mapped to that of $j^M_n((x_v)_v)$ in $\varinjlim_S H^1(\Gamma^S_F, \Delta_n^M)$. This implies that the element
$$\varprojlim_M j_n^M((x_v)_v)\in \varprojlim_M \prod_v^{'} H^1(\Gamma_v, \Delta^M_n)$$ lies in the image of the map
$$\varprojlim_M\varinjlim H^1(\Gamma_F^S, \Delta_n^M)\to \varprojlim_M\prod\nolimits^{'} H^1(\Gamma_v, \Delta_n^M).$$
\end{proof}\\
$\mathbf{Remark}.$ The sets $X(\A_F)^{n+1}_{n}$ are only defined if $X$ has a rational point and satisfies the conditions [Coh1] and [Coh2] which are required for the construction of the non-Abelian reciprocity law. However, the sets $X(\A_F)^{\mathrm{Nil}_n}$ are defined without any restriction of this kind.

\section{Appendix}
\subsubsection{The conditions [Coh1] and [Coh2] for smooth curves}
In this Appendix, we shall prove that the conditions [Coh1] and [Coh2] are satisfied if $X$ is a smooth, proper, and integral algebraic curve over $F$ (as before, we shall fix $b\in X(F)$). Recall that [Coh1] and [Coh2] are the following claims: \\
{[Coh1]: For every finite place $v$ of $F,$ $n\geq 1,$ and finite set $M$ of odd prime numbers, the Abelian group $T_n^M$ is torsion-free}.\\
{[Coh2]: For all $v, M, n$ as in [Coh1], $T_n^M$ has no non-zero $\Gamma_v$-invariants}.\\
This was claimed without proof in \cite{K}, and we hope that verifying the conditions at least for curves will shed some light on the construction. Throughout this section, we shall make use of the following observation: For any (possibly infinite) set $M$ of (odd) prime numbers, we have
$$\Delta^M_n=\prod_{l\in M} \Delta^{\{l\}}_n.$$ This follows directly from the fact that a finite nilpotent group is, in a canonical way, the product of its Sylow subgroups (see \cite{Gor}, Chapter 2, Theorem 3.5). This observation also implies that the analogous claim holds for the graded pieces $T_n^M$ of the closed lower central series of $(\Delta^{[1]})^M.$ Let us begin with the following
\begin{proposition}
For any odd prime number $l$, the $\Z_l$-module $T_n^{\{l\}}$ is torsion-free and finitely generated. \label{torsionfreeproposition}
\end{proposition}
\begin{proof}
Let $G$ be the topological fundamental group of $X(\mathbf{C}).$ It is well-known that $G$ admits a presentation 
$$G=\langle a_1,..., a_g, b_1,..., b_g\mid [a_1b_1][a_2b_2]...[a_gb_g]=e\rangle$$ for some $g\geq 0.$ Letting $G^{(n)}$ denote the lower central series of $G$, we put $G_n:=G/G^{(n+1)},$ and we let $S_n$ denote the graded pieces. We obtain a short exact sequence of finitely generated nilpotent groups
$$0\to S_n\to G_n\to G_{n-1}\to 0$$ for $n\geq 2.$ Using the claim about the decomposition of $\Delta^M_n$ from the beginning of this section, together with the fact that the geometric étale fundamental group of $X$ is isomorphic to the profinite completion of $G$, it suffices to show that the kernel of the morphism of profinite completions $\widehat{G_n}\to \widehat{G_{n-1}}$ is torsion-free. By \cite{HR}, Theorem 2.2, this kernel is simply the profinite completion of $S_n.$ By \cite{Lab}, Theorem (p. 17), the groups $S_n$ are finitely generated free Abelian groups, which implies that their pro-$l$ completions are torsion-free and finitely generated as well. 
\end{proof}\\
In order to prove that $T_n^{\{l\}}$ has no $\Gamma_v$-invariants for any finite place $v$ of $F$,  it suffices to prove the same statement for $T^{\{l\}}_n\otimes_{\Z_l}\Q_l$ since we already know that $T^{\{l\}}_n$ is torsion-free for all $n.$ This will be accomplished by first proving that $V_l(J)^{\otimes n}$ has no $\Gamma_v$-invariants for any finite place $v$ of $F$ (where $J:=\Pic^0_{C/F}$ denotes the Jacobian of $C$), and then showing that $T_n^{\{l\}}\otimes_{\Z_l}\Q_l$ embeds into $V_l(J)^{\otimes n}$ for all $n$.\\
\\
Before proceeding, recall the following notions and objects: Suppose $v$ is a finite place of $F$ and $p$ is the prime number above which $v$ lies. Given a finite-dimensional representation $V$ of $\Gamma_v$ over $\Q_p$, we have the finite-dimensional $F_{v,0}$-vector spaces $\mathrm{D}_{\mathrm{st}}(V)$ and $\mathrm{D}_{\mathrm{crys}}(V)$, where $F_{v,0}$ is the maximal absolutely unramified subfield of $F_v$ (i.e., $F_{v,0}$ is the fraction field of the ring of Witt vectors of the residue field of $F_v$). Both these vector spaces come with a natural Frobenius operator, and $\mathrm{D}_{\mathrm{st}}(V)$ comes with a natural monodromy operator $N$. In the following, we shall only consider the case where $V$ is \it semistable, \rm i.e., where $\dim_{F_{v,0}}\mathrm{D}_{\mathrm{st}}(V)=\dim_{\Q_p}V.$ See \cite{FontaineII}, particularly (5.1), pp. 155ff, for more details. If $V$ is a $\Q_l$-vector space for some $l\not=p,$ we shall consider the filtration $0\subseteq P_v\subseteq I_v\subseteq \Gamma_v,$ where $P_v$ and $I_v$ are the wild inertia subgroup and the inertia subgroup, respectively, at $v.$ We shall only consider the case where the action of all elements of $\Gamma_v$ in $V$ is unipotent, (i.e., of characteristic polynomial $(x-1)^{\dim V}$). Then $P_v$ acts trivially, and we can consider the actions of a topological generator $\tau$ of $I_v/P_v$ on $V.$ We let $N:=\log \tau,$ and we can also consider the action of the geometric Frobenius $\sigma$ on $V,$ which is analogous to the action of Frobenius on $\mathrm{D}_{\mathrm{st}}(V).$ If $V$ is the $l$-adic Tate module of an Abelian variety with semistable reduction, then $(\tau-1)^2=0,$ so $N=\tau-1$ (see \cite{SGA7}, IX, Corollaire 3.5.2(v)).
\begin{proposition}
Let $l$ be prime number. As usual, let $V_l(J):=(\varprojlim J[l^n](\overline{F}))\otimes_{\Z_l}\Q_l.$ Then $V_l(J)^{\otimes n}$ has no $\Gamma_v$-invariants for any finite place $v$ of $F$ and any $n\geq 1.$ \label{Tatemodproposition}
\end{proposition}
\begin{proof}
This follows from the ($l$-adic and $p$-adic) weight-monodromy conjecture for Abelian varieties (we shall recall all precise statements we need below). Let $v$ be a finite place of $F$. Since replacing $F_v$ by one of its finite extensions leads to replacing $\Gamma_v$ by an open subgroup (thereby making the assertion of the Proposition stronger), we may assume without loss of generality that $J$ has semiabelian reduction at $v$. Let $p$ be the prime number over which $v$ lies. We shall first treat the case where $l\not=p,$ and then indicate briefly how the same kind of argument can be used in the case $l=p.$ Consider the filtration
$$0\subseteq V_l(J)^t\subseteq V_l(J)^{I_v}\subseteq V_l(J).$$ Here, $V_l(J)^{I_v}$ stands for the inertia-invariant subspace, and $V_l(J)^t$ for the space generated by those inertia-invariant elements of the Tate module which land in the toric part of the special fibre of the Néron model. To make the notation easier, we shall write $V_{-3}:=0,$ $V_{-2}:=V_l(J)^t,$ $V_{-1}:=V_l(J)^I,$ and $V_0:=V_l(J).$ Now consider the Raynaud extension
$$0\to T\to G \to E\to 0$$ associated with the identity component of the Néron model of $J$ over the integers of $F_v$ (where $T$ and $E$ are a torus and an Abelian scheme, respectively, over the ring of integers, and $G$ is a smooth commutative group scheme; see \cite{FC}, pp. 33f, for more details). We may assume without loss of generality that $T$ is a split torus. Using Grothendieck's orthogonality theorem, which tells us that $V_l(J)^t$ is the orthogonal complement of $V_l(J^\vee)^{I_v}$ with respect to the Weil pairing (\cite{SGA7}, IX, Théorème 5.2), we see that the monodromy operator induces an isomorphism $\Gr_0 V_0\cong \Gr_{-2} V_0(-1)$ of representations of $\Gamma_v/I_v.$ Since $\Gr_{-2}V_0=V_{-2}\cong V_l(T),$ we see that $\Gr_{-2}V_0$ and $\Gr_{0}V_0$ are pure of weight $-2$ and $0$, respectively. Since $V_l(J)^{I_v}=V_l(G),$ we see that $\Gr_{-1}V_0\cong V_l(E),$ which is pure of weight $-1$ by Deligne's theorem on the Weil conjectures. The filtration on $V_0$ induces a filtration on $V_0^{\otimes n}$ for all $n\geq 0,$ whose graded pieces are given by
$$\Gr_j V_0^{\otimes n}=\bigoplus_{\alpha_1+\alpha_2+...+\alpha_n=j} \bigotimes_{\nu=1}^n\Gr_{\alpha_\nu} V_0.$$ Since the weights of $V_0$ are nonpositive, we can deduce from this that the monodromy operator is injective on $\Gr_0V_0^{\otimes n}.$ Now suppose $x\in V_0^{\otimes n}$ is $\Gamma_v$-invariant. Then it must lie in the kernel of the monodromy operator, so its image in $\Gr_0 V_0^{\otimes n}$ is trivial. Hence $x\in (V_0^{\otimes n})_{-1}.$ But all other graded pieces have strictly negative weight, so we see by induction that $x$ lies in all subspaces of the induced filtration of $V_0^{\otimes n}.$ Since this filtration is separated, it follows that $x=0.$ \\
In the case $l=p,$ we obtain a similar filtration $0=V_{-3}\subseteq V_{-2} \subseteq V_{-1} \subseteq V_0=\mathrm{D}_{\mathrm{st}}(V_p(J));$ see \cite{CI}, section 4 (as in the previous case, we extend the ground field to ensure that $J$ has semiabelian reduction). One deduces from the results of \it loc. cit. \rm (together with \cite{KM}, Corollary 1(2)) that the results about the weight filtration and the monodromy operator which we used above carry over to the case $l=p$ \it mutatis mutandis. \rm Now we observe that a non-zero $\Gamma_v$-invariant element of $V_p(J)^{\otimes n}$ gives rise to a non-zero element of $\mathrm{D}_{\mathrm{st}}(V_p(J))^{\otimes n}$ which lies in the kernel of the monodromy operator. We conclude using the same argument as above. 
\end{proof}\\
We are now ready to prove
\begin{proposition}
Let $X$ be a smooth, proper, and geometrically integral curve over $F$, and let $b\in X(F).$ Then, for all finite sets $M$ of odd primes, $T_n^M$ is torsion-free, and has no non-trivial $\Gamma_v$-invariant elements for any finite place $v$ of $F$ and any $n\geq  1$ (in other words, the conditions [Coh1] and [Coh2] are satisfied for $X$).
\end{proposition}
\begin{proof}
The first statement has already been proved in Proposition \ref{torsionfreeproposition}.  We shall now prove by induction that there is a $\Gamma_F$-equivariant surjection $(T_1^{\{l\}})^{\otimes n}\to T_n^{\{l\}}$ (where the tensor product is taken over $\Z_l$). Here, $l$ denotes any prime number. If $n=1$, the claim is trivial. Suppose we have constructed such a surjection for the natural number $n$. Consider the map
$$T_n^{\{l\}}\otimes_{\Z} T_1^{\{l\}} \to T_{n+1}^{\{l\}},$$ which sends $\alpha\otimes \beta$ to the class of the commutator $[\alpha, \beta].$ Using standard commutator identities, we see that this map is a well-defined homomorphism. Since $T_n^{\{l\}},$ $T_1^{\{l\}},$ and $T_{n+1}^{\{l\}}$ are finitely generated $\Z_l$-modules, this map factors uniquely through a $\Z_l$-homomorphism $T_n^{\{l\}}\otimes_{\Z_l} T_1^{\{l\}} \to T_{n+1}^{\{l\}}.$ We already know that any subgroup of finite index in the source is open (since the source is finitely generated over $\Z_l$), which implies that this map is automatically continuous. The image of this homomorphism is clearly dense, and since the source is compact, it follows that the homomorphism is surjective. By the induction hypothesis, we already have a surjection $(T_1^{\{l\}})^{\otimes n}\to T_n^{\{l\}},$ so the claim follows. In particular, we obtain a surjection
$$V_l(J)^{\otimes n}\to T_n^{\{l\}} \otimes_{\Z_l} \Q_l.$$ By Faltings' solution of the Tate conjecture (\cite{Fal}, Satz 3), we know that the $\Gamma_F$-representation on $V_l(J)$ is semisimple, and a theorem of Chevalley now tells us that the same is true for $V_l(J)^{\otimes n}.$ Hence the surjection above must split, and the splitting is \it a forteriori \rm $\Gamma_v$-equivariant for all finite places $v$ of $F$. The remaining part of the Proposition follows from Proposition \ref{Tatemodproposition}.
\end{proof}


\begin{thebibliography}{10}
\bibitem{CI}
Coleman, R., Iovita, A.
\textit{The Frobenius and monodromy operator for curves and Abelian varieties}. Duke Math. J., Vol. 97, No. 1 (1999)

\bibitem{Fal}
Faltings, G.
\textit{Endlichkeitss\"atze f\"ur Abelsche Variet\"aten \"uber Zahlk\"orpern}. Invent. Math. 73, pp. 349-366 (1983)

\bibitem{FC}
Faltings, G., Chai, C.-L.
\textit{Degeneration of Abelian varieties}. Ergeb. Math. Grenzgeb. (3), Springer-Verlag, 1990

\bibitem{FontaineII}
Fontaine, J.-M.
\textit{Exposé III: Représentations $p$-adiques semi-stables}. In "Périodes $p$-adiques", Astérisque 223 (1994), pp. 113-184

\bibitem{Gor}
Gorenstein, D.
\textit{Finite groups}. Chelsea Publishing Company, New, York, 1968

\bibitem{SGA7}
Grothendieck, A., Deligne, P., Katz, N. with Raynaud, M. and Rim, D. S.
\textit{Groupes de monodromie en géométrie algébrique}. Lecture Notes in Math. 269, 270, 305 (1972-73).

\bibitem{HS}
Harari, D., Stix, J.
\textit{Descent obstruction and fundamental exact sequence}. In: The Arithmetic of Fundamental Groups - PIA 2010, Contributions in Mathematical
and Computational Sciences 2, 147-166, Springer-Verlag (2012).

\bibitem{HSk}
Harari, D., Skorobogatov, A.
\textit{Non-abelian cohomology and rational points}. Compositio Math. 130, pp. 241-273 (2002)

\bibitem{HR}
Hilton, P., Roitberg, J.
\textit{Profinite Groups and Generalizations of a Theorem of Blackburn}. J. Algebra 60, pp. 289-306 (1979)

\bibitem{KM}
Katz, N., Messing, W.
\textit{Some Consequences of the Riemann Hypothesis for Abelian Varieties over Finite Fields}. Invent. Math. 23, pp. 73-77 (1974)

\bibitem{K}
Kim, M.
\textit{Diophantine geometry and non-abelian reciprocity laws I}. In "Elliptic curves, modular forms and Iwasawa theory", pp. 311-334, Zerbes, S. L. and Loeffler, D. eds, Springer Proc. Math. Stat., 188, Springer, Cham, 2016.

\bibitem{Lab}
Labute, J. P.
\textit{On the Descending Central Series of Groups with a Single Defining Relation}. J. Algebra 14, pp. 16-23 (1970)

\bibitem{N}
Neukirch, J., Schmidt, A., Wingberg, K.
\textit{Cohomology of Number fields}. Grundlehren Math. Wiss. Springer-Verlag (2013)

\bibitem{Na}
Narkiewicz, W.
\textit{ Elementary and Analytic Theory of Algebraic Numbers}. Springer-Verlag (2004)

\bibitem{Sk}
Skorobogatov, A.
\textit{Torsors and Rational Points}. Cambridge University Press (2001)

\bibitem{St}
Stix, J.
\textit{Rational points and the Arithmetic of Fundamental Groups}.  Lecture Notes in Math. Springer-Verlag (2013)

\bibitem{S}
Stoll, M.
\textit{Finite descent obstructions and rational points on curves}. Algebra and Number Theory 2, no. 5, pp. 349-391 (2008)

\bibitem{Sz}
Szamuely, T.
\textit{Galois groups and fundamental groups}. Cambridge Studies in Advanced Mathematics (2008) 














\end{thebibliography}
\end{document}